\documentclass[11pt,reqno]{amsart}
\usepackage[a4paper,left=30mm,right=30mm,top=32mm,bottom=32mm]{geometry}
\fontsize{6.5pt}{6.5pt}\selectfont 
\makeatletter
\def\@evenhead{{\fontsize{6.5pt}{6.5pt}\selectfont \hfil \leftmark\hfil\thepage}}
\def\@oddhead{{\fontsize{6.5pt}{6.5pt}\selectfont\hfil\rightmark \hfil\thepage}}
\usepackage{amsmath,amssymb,amsthm,mathtools,array}
\usepackage{extarrows}
\usepackage[dvipsnames]{xcolor}
\usepackage{graphicx}
\usepackage{booktabs,tabularx}
\usepackage{enumitem}
\setlist[enumerate]{font=\upshape}
\usepackage{tikz}
\usetikzlibrary{arrows.meta,matrix,decorations.markings}
\tikzset{
  ->-/.style={decoration={markings,mark=at position #1 with {\arrow{>}}},postaction={decorate}},
  -<-/.style={decoration={markings,mark=at position #1 with {\arrow{<}}},postaction={decorate}}
}
\usepackage{microtype}
\usepackage{needspace}
\usepackage[colorlinks=true,linkcolor=blue,citecolor=red,urlcolor=blue]{hyperref}
\usepackage{bookmark}
\usepackage{cleveref}
\numberwithin{equation}{section}
\newtheorem{theorem}{Theorem}[section]
\newtheorem{proposition}[theorem]{Proposition}
\newtheorem{lemma}[theorem]{Lemma}

\newtheorem{conjecture}[theorem]{Conjecture}
\theoremstyle{definition}
\newtheorem{definition}[theorem]{Definition}
\newtheorem{example}[theorem]{Example}
\newtheorem{remark}[theorem]{Remark}
\newcommand{\mcA}{\mathcal{A}}

\newcommand{\mcL}{\mathcal{L}}

\newcommand{\mcT}{\mathcal{T}}

\newcommand{\mfa}{\mathbf{a}}
\newcommand{\mfb}{\mathbf{b}}

\newcommand{\mfe}{\mathbf{e}}

\newcommand{\mfk}{\mathbf{k}}

\newcommand{\mfv}{\mathbf{v}}

\newcommand{\mfx}{\mathbf{x}}

\newcommand{\mfz}{\mathbf{z}}
\newcommand{\mfalpha}{\boldsymbol{\alpha}}

\newcommand{\mfeta}{\boldsymbol{\eta}}
\newcommand{\NN}{\mathbb Z_{\geq0}}
\newcommand{\ZZ}{\mathbb Z}
\newcommand{\QQ}{\mathbb Q}
\newcommand{\RR}{\mathbb R}
\newcommand{\supp}{\operatorname{supp}}
\newcommand{\cL}{\mcL}
\newcommand{\cA}{\mcA}
\newcommand{\cT}{\mcT}
\newcommand{\eps}{\varepsilon}
\newcommand{\wt}{\operatorname{wt}}
\title[Log-concavity and unimodality of cluster monomials in fan coordinates of type $A_n$]{Log-concavity and unimodality of cluster monomials\\ in fan coordinates of type $A_n$}
\author{Zhichao Chen}
\address{School of Mathematical Sciences\\ University of Science and Technology of China\\ Hefei, Anhui 230026, P. R. China}
\email{czc98@mail.ustc.edu.cn}
\author{Zhe Sun}
\address{School of Mathematical Sciences\\ University of Science and Technology of China\\ Hefei, Anhui 230026, P. R. China}
\email{sunz@ustc.edu.cn}
\date{\today}
\hypersetup{pdftitle={Log-concavity and unimodality of cluster monomials in fan coordinates of type An},pdfauthor={Zhichao Chen and Zhe Sun}}
\begin{document}
\begin{abstract}
We prove that every cluster monomial of type $A_n$ is log-concave and unimodal with respect to its Laurent expansion in a fan initial cluster. We derive a coefficient formula for laminar interval products arising from compatible diagonals in the polygon model. In even rank, we reduce log-concavity in fan coordinates to log-concavity along integer lines. In odd rank, we use convolution to handle the coefficient sums arising from the fan substitution. This gives partial affirmative answers to the log-concavity conjecture of Chen-Huang-Sun \cite{CHS26} and the unimodality conjecture of Chen \cite{Che26}. Our result also complements Shen's study \cite{She14} on the supports of products of canonical basis elements for cluster $\mathcal X$-varieties of type $A_n$.
\end{abstract}
\maketitle
{\linespread{1}\selectfont\tableofcontents}

\section{Introduction}
\subsection{Motivation}
Cluster algebras were introduced by Fomin-Zelevinsky in the study of total positivity and canonical bases \cite{FZ02}. Their generators are organized into clusters, which are related by mutations. Although mutations are rational transformations, the Laurent phenomenon guarantees that every cluster variable is a Laurent polynomial in any fixed initial cluster. Thus, it is natural to study the coefficient structure of these Laurent expansions.

Log-concavity is a central property of sequences in algebraic combinatorics, such as \cite{Bre89,Sta89,HMMS22}. Okounkov \cite{Oko03} studied log-concavity questions for multiplicities in representation theory. This provides a broader motivation for asking how positive coefficients vary with their indices. Here, we study the Laurent coefficients of cluster monomials, especially for Dynkin type $A_n$.

The linear independence of cluster monomials was conjectured in \cite{FZ04} and established using categorification and scattering diagrams \cite{IKFP13,GHKK18}. The Laurent phenomenon \cite{FZ02} and positivity for cluster algebras \cite{LS15,GHKK18} ensure that cluster monomials have nonnegative Laurent coefficients in every initial cluster. Then, Chen-Huang-Sun \cite{CHS26} proved log-concavity for cluster variables of type $A_n$ and cluster monomials of type $A_2$. Chen \cite{Che26} subsequently established both log-concavity and unimodality for cluster monomials of type $A_3$. These results lead to the question of whether both properties hold in higher rank with respect to an arbitrary initial cluster.

The cluster ensemble framework of Fock-Goncharov \cite{FG09} places these questions in a geometric setting. A cluster ensemble consists of a pair of positive spaces, the cluster $\mathcal A$- and $\mathcal X$-varieties, related by a morphism $p:\mathcal A\to\mathcal X$. Cluster variables are coordinates on the cluster $\mathcal A$-variety, while the cluster $\mathcal X$-variety carries a compatible Poisson structure. Their duality conjectures relate canonical bases of functions on one space to integral tropical points of the Langlands dual space of the other kind. In this language, our cluster monomials are universally positive regular functions on the coefficient-free $\mathcal A$-variety, and we study their Laurent coefficients in the chart associated with a fixed fan triangulation; see \Cref{def: fan}.

For surfaces, Fock-Goncharov \cite{FG06} constructed positive coordinate systems on moduli spaces of local systems. Ideal triangulations provide coordinates on the $\mathcal A$-space associated with $\mathrm{SL}_2$ and the $\mathcal X$-space associated with $\mathrm{PGL}_2$. A disk with $n+3$ marked boundary points gives the polygon combinatorics of type $A_n$. Specializing the boundary $\mathcal A$-coordinates to $1$ gives the coefficient-free convention used here. Goncharov-Shen \cite{GS15} further related positive potentials on decorated local-system spaces to canonical bases: for a disk, integral tropical points satisfying the potential inequalities parametrize canonical bases in tensor product invariants for the Langlands dual group. These constructions connect positivity and tropical geometry with the noncrossing combinatorics used below. Related geometric applications include the study of exponential volumes of moduli spaces of hyperbolic surfaces by Goncharov-Sun \cite{GS26}.

For cluster $\mathcal X$-varieties of type $A_n$, Shen \cite{She14} proved that the supports of products of canonical basis elements are Stasheff polytopes in the tropical $\mathcal A$-space. This gives a geometric description of which basis elements occur in a product. Our result addresses a related coefficient question for the cluster $\mathcal A$-variety. In a fan chart, the Laurent coefficients of every cluster monomial are log-concave and unimodal along each coordinate direction. We explain the common polygon combinatorics and the distinction between these two types of expansion in Remark~\ref{rem:shen}.

In this paper, we assume that the initial triangulation is a fan, whose quiver is a linearly oriented path of type $A_n$ up to relabeling and reversal of all arrows. 
The main difficulty is that coordinatewise log-concavity need not be preserved under multiplication of multivariate polynomials \cite[Remark~3.9]{Che26}. Hence, the result for individual cluster variables does not imply the corresponding statement for cluster monomials. However, variables in a common cluster correspond to compatible diagonals. Their fan expansions involve intervals that are nested or disjoint, so we can use this laminar structure to count the coefficients of their product.

\subsection{Main result}
Throughout the paper, we work with skew-symmetric and coefficient-free cluster algebras. Unless otherwise specified, log-concavity of a Laurent polynomial refers to its coefficients along each coordinate direction, with all other exponents fixed. Unimodality is understood in the same way and precise definitions are given in Section~\ref{sec:sequences}.

Chen-Huang-Sun conjectured that cluster monomials of type $A_n$ are log-concave \cite[Conjecture~5.5]{CHS26}. Afterwards, Chen \cite{Che26} extended this question to unimodality and independence of the initial seed. We state the two questions below.
\begin{conjecture}[{\cite[Conjecture~5.5]{CHS26}, \cite[Conjecture~5.11]{Che26}}]\label{conj:cluster-monomials}
Let $n\geq4$.
\begin{enumerate}
\item Cluster monomials of type $A_n$ are log-concave and unimodal.
\item For each cluster monomial of type $A_n$, these properties are independent of the choice of the initial seed.
\end{enumerate}
\end{conjecture}

We call a family of nonempty sets \emph{laminar} if any two members are disjoint or one contains the other; see Definition~\ref{def:laminar}. A \emph{laminar interval family} is a laminar family whose members are integer intervals. The following proposition expresses cluster monomials in terms of laminar interval products.
\begin{proposition}[{Proposition~\ref{prop:reduction}}]\label{prop:intro-reduction}
Let $\mfx=(x_1,\ldots,x_n)$ be the fan initial cluster of type $A_n$. For every cluster monomial $M$, there exist $\mfeta\in\NN^n$, a laminar interval family $\cL$ in $\{1,\ldots,n+1\}$, and nonnegative integers $a_I$ such that
\[
 M=\mfx^{\mfeta}\Phi(P),\quad
 P(\mfz)=\prod_{I\in\cL}\left(\sum_{i\in I}z_i\right)^{a_I}.
\]
Here, the substitution map
$
 \Phi(z_i)=x_{i-1}^{-1}x_i^{-1},
$ where $1\leq i\leq n+1$ and $x_0=x_{n+1}=1$.
\end{proposition}

When the initial triangulation is a fan, we obtain the following result for all cluster monomials in arbitrary rank with the help of \Cref{prop:intro-reduction}.
\begin{theorem}\label{thm:main}
Let $\cA$ be a cluster algebra whose initial exchange matrix corresponds to a linearly oriented quiver of type $A_n$. Then, every cluster monomial of $\cA$ is log-concave and unimodal with respect to its initial cluster $\mfx=(x_1,\ldots,x_n)$. Moreover, every nonzero coordinate slice of its Laurent coefficient array has no internal zeros.
\end{theorem}
Thus, Theorem~\ref{thm:main} proves the conjecture in \cite[Conjecture~5.5]{CHS26} and part~(1) of Conjecture~\ref{conj:cluster-monomials} for fan initial triangulations. We do not establish the seed independence asserted in part~(2). The distinction between changing the initial seed and transporting an expansion under an isomorphism is explained in Section~\ref{sec:scope}.

We prove the main theorem in Section~\ref{sec:mainproof} using Proposition~\ref{prop:intro-reduction}. We apply the coefficient formula in Proposition~\ref{prop:coeff} to the homogeneous polynomial $P$ and analyze how its coefficients combine under $\Phi$. This substitution sends $\mfz^{\mfk}$ to $\mfx^{-L(\mfk)}$, where
\[
 L(k_1,\ldots,k_{n+1})=(k_1+k_2,\ldots,k_n+k_{n+1}).
\]
Up to the fixed exponent shift $\mfeta$, the Laurent coefficients of $M$ are therefore sums of coefficients of $P$ over the fibers of $L$.

We analyze these fibers on vectors of fixed total degree. For even $n$, the map $L$ is injective on such vectors, so we reduce the problem to log-concavity of the coefficients of $P$ along integer lines. For odd $n$, each slice either has at most one nonzero term or, after a shift and reversal of the index, takes the form
\[
 C(q)=H(q)\sum_{t\in\ZZ}A(t)B(t+q),
\]
where $A(t)$, $B(u)$, and $H(q)$ are nonnegative log-concave sequences with interval support. Moreover, $A(t)$ and $B(u)$ have finite support. We recognize the sum as the convolution of $A(-t)$ and $B(u)$ and apply closure under convolution and pointwise multiplication to obtain the required log-concavity.

\begin{remark}[Relation with Shen's work]\label{rem:shen}
In Shen's expansion \cite{She14}, the indices label canonical basis elements on the cluster $\mathcal X$-variety and lie in the tropical $\mathcal A$-space. In our expansion, the indices are Laurent exponent vectors in a fixed chart of the cluster $\mathcal A$-variety. Thus, the two results concern different coefficient arrays: Shen describes the support of multiplication in a canonical basis, whereas we establish inequalities among the Laurent coefficients of individual cluster monomials. The ensemble framework relates the underlying spaces, but does not identify these coefficient arrays.

The polygon model gives a concrete combinatorial connection. Shen's construction uses integral laminations represented by weighted noncrossing edges and diagonals, with conditions on their weights. In our proof, compatible diagonals give nested or disjoint intervals in the fan expansions. This laminar structure determines both the support and the coefficients of an auxiliary product in Proposition~\ref{prop:coeff}. Tracking these coefficients under the fan substitution then gives log-concavity and interval support for the coordinate slices. Our results therefore complement the geometric study of supports with coefficient inequalities in fan coordinates on the cluster $\mathcal A$-variety.
\end{remark}

\subsection{Outline of the paper}
In Section~\ref{sec:prelim}, we fix the notation and recall cluster algebras and their polygon realization in type $A_n$. In Section~\ref{sec:sequences}, we introduce log-concavity and unimodality for sequences and Laurent polynomials. Then, we recall the closure properties needed in the proof. In Section~\ref{sec:laminar}, we give a coefficient formula for laminar products and establish log-concavity along integer lines. In Section~\ref{sec:fan}, we express cluster monomials in fan coordinates and describe the corresponding exponent fibers. In Section~\ref{sec:mainproof}, we prove Theorem~\ref{thm:main}. In Section~\ref{sec:examples}, we give explicit examples of rank $4$ and $5$. In Section~\ref{sec:scope}, we discuss seed independence and its relation to tropical geometry and canonical bases.

\section{Preliminaries}\label{sec:prelim}
We first fix the notation and recall the basic definitions of cluster algebras. We then introduce the polygon model of type $A_n$ and the fan initial cluster.
\subsection{Notation}
We write $\NN=\{0,1,2,\ldots\}$ and, for any positive integer $n$, $[n]=\{1,\ldots,n\}$. A subset $I\subseteq\ZZ$ is an \emph{integer interval} if $p,r\in I$ and $p\leq q\leq r$ imply $q\in I$ for every integer $q$. We allow the empty set, singletons, and unbounded intervals. For integers $a,b$, we write $[a,b]=\{i\in\ZZ:a\leq i\leq b\}$; this interval is empty when $a>b$. For a real number $r$, let $[r]_+=\max\{r,0\}$. Empty sums and products are taken to be $0$ and $1$, respectively, and $0!=1$.

We write vectors and tuples in boldface, such as $\mathbf{k}$, $\mathbf{b}$, and $\boldsymbol{\eta}$, while their scalar coordinates are denoted by $k_i$, $b_i$, and $\eta_i$, respectively. For $\mfalpha=(\alpha_1,\ldots,\alpha_n)\in\ZZ^n$ and a tuple of variables $\mfx=(x_1,\ldots,x_n)$, we write
\[
 \mfx^{\mfalpha}=\prod_{i=1}^n x_i^{\alpha_i},
 \quad \wt(\mfalpha)=\sum_{i=1}^n\alpha_i.
\]
We denote by $S_n$ the group of permutations of $[n]$ and denote by $\mfe_j$ the $j$th standard basis vector. If $f$ is a Laurent polynomial in $\mfx$, then $[\mfx^{\mfalpha}]f$ denotes its coefficient at $\mfx^{\mfalpha}$, and
\[
 \supp(f)=\{\mfalpha\in\ZZ^n:[\mfx^{\mfalpha}]f\neq0\}.
\]
We use the same monomial, weight, and coefficient notation for tuples of any finite length, in particular for the auxiliary variables $\mfz$ introduced below. We set coefficients outside the support equal to zero. 

\subsection{Seeds and cluster monomials}
We recall seeds and mutation in the setting fixed above, following \cite{FZ02, FZ03, Nak23}.

\begin{definition}[Seed and mutation]
Let $\mathcal F$ be a rational function field of transcendence degree $n$ over $\QQ$. A \emph{seed} is a pair $\Sigma=(\mfx,B)$, where $\mfx=(x_1,\ldots,x_n)$ is a transcendence basis generating $\mathcal F$ and $B=(b_{ij})$ is a skew-symmetric integer matrix. We call $\mfx$ a \emph{cluster}, its entries \emph{cluster variables}, and $B$ the \emph{exchange matrix}. The number $n$ of cluster variables is called the \emph{rank} of the seed.

For any $k\in[n]$, the \emph{mutation of $\Sigma$ in the direction $k$} is $\mu_k(\Sigma)=(\mfx',B')$, where
\begin{equation}\label{eq:mutation}
 x_i'=\begin{cases}
 x_i,&i\neq k,\\[2pt]
 \displaystyle\frac{\prod_{j=1}^n x_j^{[b_{jk}]_+}
                 +\prod_{j=1}^n x_j^{[-b_{jk}]_+}}{x_k},&i=k,
 \end{cases}
\end{equation}
and
\begin{equation}\label{eq:Bmutation}
 b_{ij}'=\begin{cases}
 -b_{ij},&i=k\text{ or }j=k,\\
 b_{ij}+[b_{ik}]_+b_{kj}+b_{ik}[-b_{kj}]_+,&\text{otherwise}.
 \end{cases}
\end{equation}
In fact, $(\mfx',B')$ is again a seed and $\mu_k$ is involutive, that is $\mu_k(\mfx',B')=(\mfx,B)$; see also \cite{Nak23}. A seed obtained by a finite sequence of mutations is said to be in the \emph{mutation equivalence class} of $\Sigma$.
\end{definition}

\begin{definition}[Cluster algebra and cluster monomial]
The \emph{coefficient-free cluster algebra} $\cA=\cA(\Sigma)$ is the $\ZZ$-subalgebra of $\mathcal F$ generated by the variables in all seeds in the mutation equivalence class of $\Sigma$. The \emph{rank} of $\cA$ is $n$. A \emph{cluster monomial} is an element
\[
 M=y_1^{a_1}\cdots y_n^{a_n},\quad a_1,\ldots,a_n\in\NN,
\]
where $(y_1,\ldots,y_n)$ is any cluster.
\end{definition}

In fact, skew-symmetric integer matrices of size $n$ are in one-to-one correspondence with quivers on the vertex set $[n]$ without loops or oriented $2$-cycles. We associate $B$ with the quiver having $b_{ij}$ arrows from $i$ to $j$ whenever $b_{ij}>0$. The cluster algebra is of \emph{type $A_n$} if $B$ belongs to the mutation equivalence class of the exchange matrix of a path on $n$ vertices with one arrow on each edge. By the Laurent phenomenon \cite[Theorem~3.1]{FZ02}, every cluster variable with respect to the initial cluster $\mfx=(x_1,\dots,x_n)$ belongs to
\[
 \ZZ[x_1^{\pm1},\ldots,x_n^{\pm1}].
\]
Moreover, the positivity in \cite[Theorem~1.1]{LS15} and \cite[Theorem~4.10]{GHKK18} ensure that all its Laurent coefficients are nonnegative. Therefore, every cluster monomial has a Laurent expansion with nonnegative integer coefficients in any fixed initial cluster.
\begin{remark}
In the terminology of Fock-Goncharov \cite{FG09}, the preceding positivity statement says that cluster monomials are \emph{universally positive}: their Laurent expansions have nonnegative integer coefficients in every cluster chart of the $\mathcal A$-variety. In type $A_n$, cluster monomials form the theta basis of Gross-Hacking-Keel-Kontsevich \cite{GHKK18}. Hence, Theorem~\ref{thm:main} establishes further properties of their Laurent coefficients in fan charts.
\end{remark}

\subsection{The polygon model}
We fix a convex polygon with vertices $0,1,\ldots,n+2$ in cyclic order. A \emph{boundary edge} joins consecutive vertices, where $0$ and $n+2$ are also consecutive. A \emph{diagonal} joins two nonconsecutive vertices. 
\begin{definition}[Compatible diagonals and triangulations]\label{def:compatible}
Two distinct diagonals are \emph{compatible} if their interiors do not intersect (they may share endpoints). A family of diagonals is \emph{compatible} if every two distinct members are compatible. A \emph{triangulation} is a maximal compatible family of diagonals.
\end{definition}
Note that every triangulation of the $(n+3)$-gon contains exactly $n$ diagonals; see \Cref{fig:a4-fan-quiver}.

\begin{proposition}[{\cite{FZ03,FST08,Sch08}}]\label{prop:polygon}
Cluster variables of type $A_n$ correspond to diagonals of the $(n+3)$-gon, and clusters correspond to triangulations. Mutation corresponds to flip replacing one diagonal of a quadrilateral by the other. We write $p_{ab}=p_{ba}$ for the variable associated with a diagonal joining vertices $a,b$ and set $p_{ab}=1$ for boundary edges. If $a,b,c,d$ are four vertices in cyclic order, then
\begin{equation}\label{eq:ptolemy}
 p_{ac}p_{bd}=p_{ab}p_{cd}+p_{ad}p_{bc}.
\end{equation}
\end{proposition}

\begin{figure}[htbp]
\centering
\begin{tikzpicture}[line cap=round,line join=round]
  \coordinate (a) at (-2.2,-1.05);
  \coordinate (b) at ( 1.9,-1.05);
  \coordinate (c) at ( 2.25, 1.15);
  \coordinate (d) at (-1.55, 1.45);
  \draw[thick] (a)--node[below,font=\small] {$p_{ab}$}(b)
               --node[right,font=\small] {$p_{bc}$}(c)
               --node[above,font=\small] {$p_{cd}$}(d)
               --node[left,font=\small] {$p_{ad}$}(a);
  \draw[blue!70!black,very thick] (a)--
    node[pos=.27,above left,font=\small,fill=white,inner sep=1pt] {$p_{ac}$}(c);
  \draw[red!70!black,very thick] (b)--
    node[pos=.27,above right,font=\small,fill=white,inner sep=1pt] {$p_{bd}$}(d);
  \foreach \v in {a,b,c,d}
    \fill (\v) circle (1.5pt);
  \node[below left,font=\small] at (a) {$a$};
  \node[below right,font=\small] at (b) {$b$};
  \node[above right,font=\small] at (c) {$c$};
  \node[above left,font=\small] at (d) {$d$};
\end{tikzpicture}
\caption{The Ptolemy relation for cyclically ordered vertices $a,b,c,d$. The product $p_{ac}p_{bd}$ of the crossing diagonals equals the sum $p_{ab}p_{cd}+p_{ad}p_{bc}$ of the products of opposite sides.}
\label{fig:ptolemy-quadrilateral}
\end{figure}
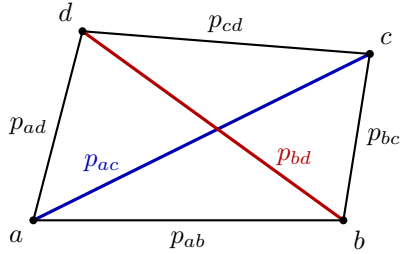
Figure~\ref{fig:ptolemy-quadrilateral} illustrates \eqref{eq:ptolemy}. Note that boundary edges are specialized to $1$ and do not represent additional cluster variables. Hence, for any triangulation $\cT$, we can write a monomial in the associated cluster as
\begin{equation}\label{eq:polygonmonomial}
 M=\prod_{(a,b)\in\cT}p_{ab}^{r_{ab}},\quad r_{ab}\in\NN.
\end{equation}

\begin{definition}[Fan triangulation]\label{def: fan}
A \emph{fan triangulation} consists of all diagonals incident with one fixed vertex. We denote by $(a,b)$ the diagonal joining vertices $a$ and $b$. Here, $(a,b)$ and $(b,a)$ denote the same diagonal. After relabeling the vertices, we take the common vertex to be $0$ and write
\begin{equation}\label{eq:fantri}
 \cT_0=\{(0,2),(0,3),\ldots,(0,n+1)\},
 \quad x_i=p_{0,i+1}\quad(i\in[n]).
\end{equation}
The initial cluster $\mfx$ in \eqref{eq:fantri} is called the \emph{fan initial cluster} or \emph{fan coordinates}. 

For example, \Cref{fig:a4-fan-quiver} represents a fan triangulation and $(0,2)$ denotes the diagonal from vertex $0$ to vertex $2$, whose cluster variable is $p_{0,2}=x_1$. The associated quiver is of type $A_4$.
\end{definition}

\begin{figure}[ht]
\centering
\begin{tikzpicture}[
  line cap=round,
  line join=round,
  diagonal/.style={blue!70!black,very thick},
  qvertex/.style={circle,draw,fill=white,minimum size=5.4mm,inner sep=0pt,font=\small},
  qarrow/.style={-{Stealth[length=2.2mm]},thick}
]
  \coordinate (c) at (0,0);
  \foreach \k in {0,...,6}
    \coordinate (v\k) at ({180-360*\k/7}:2.05cm);
  \draw[thick] (v0)--(v1)--(v2)--(v3)--(v4)--(v5)--(v6)--cycle;
  \foreach \k [evaluate=\k as \i using int(\k-1)] in {2,...,5}
    \draw[diagonal] (v0)--node[pos=.54,font=\small,fill=white,inner sep=1pt] {$x_\i$}(v\k);
  \foreach \k in {0,...,6}
    \node[font=\small] at ({180-360*\k/7}:2.38cm) {$\k$};
  \begin{scope}[xshift=5.2cm]
    \foreach \i in {1,...,4}
      \node[qvertex] (q\i) at ({1.42*(4-\i)},0) {$x_\i$};
    \foreach \i [evaluate=\i as \j using int(\i+1)] in {1,...,3}
      \draw[qarrow] (q\j)--(q\i);
  \end{scope}
\end{tikzpicture}
\caption{The fan triangulation of the heptagon and its linearly oriented quiver of type $A_4$.}
\label{fig:a4-fan-quiver}
\end{figure}
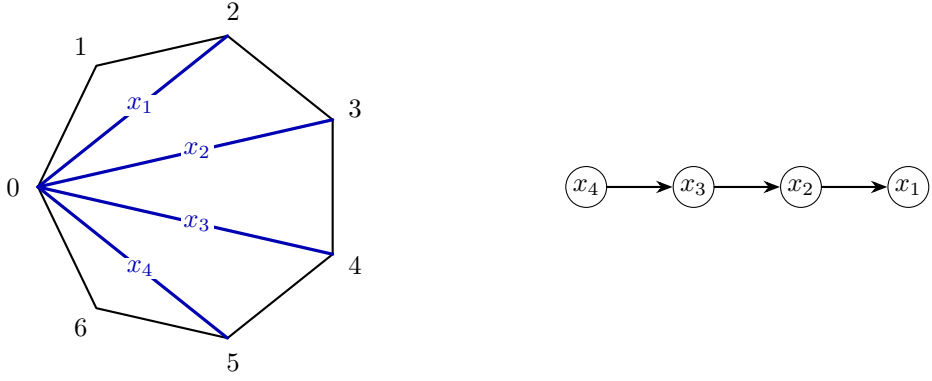

\section{Log-concavity, unimodality, and convolution}\label{sec:sequences}
We recall basic facts about log-concave sequences and convolution based on \cite{Sta89,Bre89}. We then review log-concavity \cite{CHS26} and unimodality \cite{Che26} for cluster Laurent polynomials, formulated here in terms of coordinate slices.
\subsection{Sequences and interval support}
Throughout this paper, $\mfa=\{a_q\}_{q\in\ZZ}$ denotes a sequence of nonnegative real numbers. Its support is $\supp(\mfa)=\{q\in\ZZ:a_q>0\}$. We say that $\mfa$ \emph{has interval support} if its support is an integer interval, or equivalently, if $a_p,a_r>0$ with $p<r$ implies $a_q>0$ for every integer $p<q<r$. We also express this property by saying that $\mfa$ \emph{has no internal zeros}. The zero sequence and sequences supported at a single index have interval support under our convention. A sequence has \emph{finite support} if it has only finitely many nonzero entries. Hence, a nonzero sequence with finite interval support is positive exactly on some $[a,b]$ with $a\leq b$.

\begin{definition}[Log-concavity and unimodality]
A nonnegative sequence $\mfa=\{a_q\}_{q\in\ZZ}$ is \emph{log-concave} if
\begin{equation}\label{eq:seqLC}
 a_q^2\geq a_{q-1}a_{q+1}\quad(q\in\ZZ).
\end{equation}
A nonzero sequence of finite support is \emph{unimodal} if there exists an integer $q_0$ such that it is nondecreasing for $q\leq q_0$ and nonincreasing for $q\geq q_0$. The zero sequence is regarded as unimodal.
\end{definition}
We will apply these definitions to the coefficient sequences obtained by fixing all but one exponent of a Laurent polynomial.

\begin{lemma}\label{lem:ratio}
Let $\mfa=\{a_q\}_{q\in\ZZ}$ be a nonnegative log-concave sequence with interval support. Whenever $a_q,a_{q+1},a_{q+2}>0$, we have
\[
 \frac{a_{q+1}}{a_q}\geq\frac{a_{q+2}}{a_{q+1}}.
\]
If $\mfa$ has finite support, then it is unimodal.
\end{lemma}
\begin{proof}
We obtain the ratio inequality by dividing $a_{q+1}^2\geq a_q a_{q+2}$ by $a_q a_{q+1}$. If the support is finite, the successive ratios are nonincreasing, so the positive terms cannot increase again after they start to decrease. Hence the sequence is unimodal. In particular, the zero and singleton cases are immediate.
\end{proof}

\begin{lemma}\label{lem:closures}
Let $\mfa,\mfa^{(1)},\ldots,\mfa^{(s)}$ be nonnegative log-concave sequences with interval support. For any $h\in\ZZ$, each of the sequences
\[
 \{a_{q-h}\}_{q\in\ZZ},\quad
 \{a_{-q}\}_{q\in\ZZ},\quad
 \{\prod_{r=1}^s a^{(r)}_q\}_{q\in\ZZ}
\]
is log-concave and has interval support.
\end{lemma}
\begin{proof}
Note that shifting or reversing the index preserves the log-concavity inequalities and interval support. For the pointwise product, we multiply the inequalities for the individual sequences. Its support is the intersection of their supports and is therefore an integer interval again.
\end{proof}

\subsection{Convolution of sequences}
We next recall that convolution preserves log-concavity for nonnegative sequences with finite interval support.

\begin{definition}[Convolution and reversal]
For sequences $\mfa,\mfb$ of finite support, their \emph{convolution} is
\[
 (\mfa*\mfb)_q=\sum_{t\in\ZZ}a_t b_{q-t}.
\]
We define their \emph{correlation} to be $\{\sum_t a_t b_{t+q}\}_{q\in\ZZ}$ and write $\check{\mfa}=\{\check{a}_t\}_{t\in\ZZ}$ for the \emph{reversal} of $\mfa$, where $\check{a}_t=a_{-t}$. Then,
\[
 \sum_{t\in\ZZ}a_t b_{t+q}=(\check{\mfa}*\mfb)_q.
\]
Note that all these sums are finite.
\end{definition}

\begin{proposition}[{\cite[Proposition~2]{Sta89}}]\label{prop:convolution}
If $\mfa,\mfb$ are nonnegative log-concave sequences of finite support with no internal zeros, then $\mfa*\mfb$ has the same properties. In particular, it is unimodal.
\end{proposition}
In fact, if the supports $\supp(\mfa)=[l_1,r_1]$ and $\supp(\mfb)=[l_2,r_2]$ are nonempty, then $\supp(\mfa*\mfb)=[l_1+l_2,r_1+r_2]$.
For example, we take $\mfa=(1,2,1)$ on $[0,2]$ and $\mfb=(1,1)$ on $[0,1]$, with both sequences zero elsewhere. Their convolution $\mfa*\mfb=(1,3,3,1)$ is log-concave and has support $[0,3]$.

%\Needspace{14\baselineskip}
\subsection{Coordinate slices of Laurent polynomials}
We now formulate log-concavity and unimodality for Laurent polynomials in terms of their coordinate slices.

\begin{definition}[Coordinate slice]\label{def:slice}
We write a Laurent polynomial with nonnegative coefficients as
\[
 f(\mfx)=\sum_{\mfalpha\in\ZZ^n}c_{\mfalpha}\mfx^{\mfalpha},
 \quad c_{\mfalpha}\geq0.
\]
We fix a direction $j\in[n]$ and integer exponents $\alpha_i$ for all $i\neq j$. The corresponding \emph{coordinate slice in direction $j$} is the sequence $\mfa=\{a_q\}_{q\in\ZZ}$ defined by
\[
 a_q=\bigl[x_1^{\alpha_1}\cdots x_{j-1}^{\alpha_{j-1}}
             x_j^q x_{j+1}^{\alpha_{j+1}}\cdots x_n^{\alpha_n}\bigr]f.
\]
Equivalently, $a_q=c_{\mfalpha(q)}$, where $\mfalpha(q)=(\alpha_1,\ldots,\alpha_{j-1},q,\alpha_{j+1},\ldots,\alpha_n)$. Note that this sequence is indexed by all of $\ZZ$ and has finite support, with $a_q=0$ whenever the corresponding monomial is absent.
\end{definition}

\begin{definition}[Log-concavity and unimodality of Laurent polynomials]\label{def:multivariate}
We call $f$ \emph{log-concave} if every coordinate slice is log-concave. More precisely, this means
\begin{equation}\label{eq:polyLC}
 c_{\mfalpha}^2\geq c_{\mfalpha-\mfe_j}c_{\mfalpha+\mfe_j}
 \quad(\mfalpha\in\ZZ^n,\ j\in[n]).
\end{equation}
We call $f$ \emph{unimodal} if, for every nonzero coordinate slice $\mfa=\{a_q\}_{q\in\ZZ}$, we can choose an integer $q_0$ such that
\[
 a_q\leq a_{q+1}\quad(q<q_0),
 \quad a_q\geq a_{q+1}\quad(q\geq q_0).
\]
Note that we choose $q_0$ after fixing the other exponents. Thus, different slices may have different modes. We say that $f$ has \emph{no internal zeros in coordinate directions} if every coordinate slice has interval support. By Laurent phenomenon \cite{FZ02}, a cluster monomial has any of these properties with respect to an initial cluster if its Laurent expansion in that cluster has the property.
\end{definition}

\begin{example}\label{ex:a3-cluster-slices}
We illustrate how to read coordinate slices from a Laurent expansion. We take the fan initial triangulation $\cT_0=\{(0,2),(0,3),(0,4)\}$ of type $A_3$, with $x_i=p_{0,i+1}$. The diagonals $\cT=\{(1,5),(2,4),(2,5)\}$ form a different triangulation. We consider the cluster monomial $M=p_{15}p_{24}p_{25}$ in the cluster associated with $\cT$ and expand it in the fan variables $\mfx$. By the fan expansion formula proved in Lemma~\ref{lem:fanexp}, we obtain
\[
 M=\frac{(1+x_2)(x_1+x_3)^2(x_1+x_3+x_1x_2)}{x_1x_2^3x_3^2}.
\]
Then, grouping the numerator by powers of $x_1$ and $x_3$, we obtain
\begin{align*}
 x_1x_2^3x_3^2M={}&x_1^3(1+2x_2+x_2^2)
   +x_1^2x_3(3+5x_2+2x_2^2)\\
  &+x_1x_3^2(3+4x_2+x_2^2)+x_3^3(1+x_2).
\end{align*}
Hence, the nonzero slices in the $x_2$ direction are given in Table~\ref{tab:a3-coordinate-slices}. We list the coefficients at exponents $-3,-2,-1$ for each fixed pair $(\alpha_1,\alpha_3)$.
\begin{table}[htbp]
\centering
\caption{Nonzero coordinate slices of $M$ in the $x_2$ direction.}
\label{tab:a3-coordinate-slices}
\begin{tabular}{cc}
\toprule
$(\alpha_1,\alpha_3)$ & $(a_{-3},a_{-2},a_{-1})$ \\
\midrule
$(2,-2)$ & $(1,2,1)$ \\
$(1,-1)$ & $(3,5,2)$ \\
$(0,0)$ & $(3,4,1)$ \\
$(-1,1)$ & $(1,1,0)$ \\
\bottomrule
\end{tabular}
\end{table}
Each listed slice is log-concave with finite interval support, which implies unimodality by Lemma~\ref{lem:ratio}. All other slices in this direction are zero. Moreover, the $x_1$ and $x_3$ exponents of every term of $M$ sum to zero. Therefore, each slice in either of those directions has at most one nonzero term.
\end{example}

\begin{lemma}\label{lem:monomialshift}
Multiplication by a Laurent monomial with positive coefficient and permutation of the variables preserve log-concavity, unimodality, and the absence of internal zeros in coordinate directions. A nonnegative log-concave Laurent polynomial with no internal zeros in coordinate directions is unimodal.
\end{lemma}
Indeed, multiplication by a Laurent monomial translates the coefficient array and scales it by a positive constant, while a permutation of the variables permutes the coordinate directions. Hence, both operations preserve all three properties. The last assertion follows from Lemma~\ref{lem:ratio}, since every coordinate slice has finite support.

\section{Coefficients of laminar products}\label{sec:laminar}
We now derive a coefficient formula for products indexed by a laminar family. We then use this formula to establish log-concavity along arbitrary integer lines of exponent vectors.
\subsection{Laminar families and coefficient support}
We first introduce laminar families and then describe the support and coefficients of the associated products. Throughout this section, $m$ is an arbitrary positive integer denoting the number of auxiliary variables. In the application to type $A_n$ in Section~\ref{sec:fan}, we take $m=n+1$.

\begin{definition}[Laminar family]\label{def:laminar}
A family $\cL$ of nonempty subsets of $[m]$ is \emph{laminar} (see also \cite[Section~13.4]{Sch03}) if, for any $I,J\in\cL$, at least one of the following conditions holds:
\[
 I\cap J=\varnothing\quad\text{or}\quad I\subseteq J\quad\text{or}\quad J\subseteq I.
\]
It is a \emph{laminar interval family} if every member is an integer interval. We regard $\cL$ as a finite collection of distinct sets and record multiplicities separately by exponents.
\end{definition}

\begin{example}\label{ex:laminar-family}
We illustrate the distinction between containment, disjointness, and a partial overlap. On $[6]$, the family
\[
 \cL=\{[1,6],[1,3],[2,3],[4,6]\}
\]
is a laminar interval family. Indeed, we have $[2,3]\subset[1,3]\subset[1,6]$, while $[4,6]\subset[1,6]$ is disjoint from both $[1,3]$ and $[2,3]$. Hence, every pair is nested or disjoint. On the other hand, $\{[1,3],[3,5]\}$ is not laminar, since the two intervals intersect at $3$ and neither contains the other. Note that our intervals consist of integers, so a shared endpoint gives a nonempty intersection.
\end{example}

We fix a laminar family $\cL$ and exponents $a_I\in\NN$ for $I\in\cL$. If necessary, we add $[m]$ with exponent zero. Then, every member of $\cL$ is contained in a common largest set, while laminarity and the product remain unchanged. For the variables $\mfz=(z_1,\ldots,z_m)$, we define
\begin{equation}\label{eq:P}
 P(\mfz)=\prod_{I\in\cL}\left(\sum_{i\in I}z_i\right)^{a_I},
 \quad d=\sum_{I\in\cL}a_I.
\end{equation}
Then, $P$ is homogeneous of degree $d$. Note that $d=0$ is allowed and gives $P=1$.

For each $I\in\cL$ and $\mfk=(k_i)\in\ZZ^m$, we define
\begin{equation}\label{eq:Dh}
 D_I=\sum_{J\in\cL,\ J\subseteq I}a_J,
 \quad \wt_I(\mfk)=\sum_{i\in I}k_i,
 \quad h_I(\mfk)=\wt_I(\mfk)-D_I.
\end{equation}
Thus, $D_I$ is the total degree contributed by factors indexed by subsets of $I$, while $h_I(\mfk)$ measures the excess of $\wt_I(\mfk)$ over this degree. We set
\begin{equation}\label{eq:Omega}
 \Omega=\left\{\mfk\in\NN^m:
       \wt(\mfk)=d,\ h_I(\mfk)\geq0\text{ for all }I\in\cL\right\}.
\end{equation}
We call the vectors in $\Omega$ \emph{admissible}. Note that $D_{[m]}=d$, so $h_{[m]}(\mfk)=0$ whenever $\wt(\mfk)=d$. Hence, the inequality for $[m]$ imposes no additional restriction.

\begin{proposition}[Laminar coefficient formula]\label{prop:coeff}
The support of $P$ is exactly $\Omega$. For every $\mfk\in\Omega$,
\begin{equation}\label{eq:coeff}
 [\mfz^{\mfk}]P
 =\frac1{\prod_{i=1}^m k_i!}
   \prod_{I\in\cL}\frac{(h_I(\mfk)+a_I)!}{h_I(\mfk)!}.
\end{equation}
\end{proposition}
\begin{proof}
We write $P(\mfz)$ as
\[
 P=\prod_{I\in\cL}\prod_{r=1}^{a_I}
       \left(\sum_{i\in I}z_i\right),
\]
and label each linear factor by $(I,r)$. The coefficient $[\mfz^{\mfk}]P$ counts the choices of one summand from each factor in which $z_i$ is chosen exactly $k_i$ times.

Any such choice gives $k_i\geq0$ and $\wt(\mfk)=d$. For each $I\in\cL$, the factors indexed by $(J,r)$ with $J\subseteq I$ contribute exactly $D_I$ variables with indices in $I$. Hence, $\wt_I(\mfk)\geq D_I$, proving $\operatorname{supp}(P)\subseteq\Omega$.

We now fix $\mfk\in\Omega$. The target monomial $\mfz^{\mfk}$ contains $k_i$ occurrences of $z_i$, which we distinguish by labels $z_i^{(1)},\ldots,z_i^{(k_i)}$. Thus, we have $\sum_i k_i=d$ labeled copies in total. We count assignments of one copy to each linear factor $(I,r)$, using no copy more than once and allowing $z_i^{(s)}$ only when $i\in I$.

We assign copies to the factors associated with each $I\in\cL$ only after we have done so for all its proper subsets in $\cL$. By laminarity, any other set treated earlier is disjoint from $I$. Thus, the number of unused copies with indices in $I$ is
\[
 \wt_I(\mfk)-\sum_{J\in\cL,\ J\subsetneq I}a_J
 =h_I(\mfk)+a_I\geq a_I.
\]
Since all $a_I$ factors are equal to $\sum_{i\in I}z_i$, we can assign any remaining copy with index in $I$ to any of these factors. We therefore choose $a_I$ distinct copies and assign them to $(I,1),\ldots,(I,a_I)$ in
\[
 \frac{(h_I(\mfk)+a_I)!}{h_I(\mfk)!}
\]
ways, with value $1$ when $a_I=0$. This count is positive and does not depend on which copies were used earlier. At this step we use exactly $a_I$ copies, leaving the others available for later factors. After all sets have been treated, we have used $\sum_I a_I=d$ copies. Since the target monomial also contains $d$ copies, none remain. Multiplying the counts gives the total number of assignments
\[
 \prod_{I\in\cL}\frac{(h_I(\mfk)+a_I)!}{h_I(\mfk)!}.
\]

After forgetting the labels, each assignment gives a choice of summands with product $\mfz^{\mfk}$. Each such choice corresponds to exactly $\prod_{i=1}^m k_i!$ assignments of the labeled copies. Indeed, for each $i$, we can freely permute the $k_i$ copies among the factors contributing $z_i$, and these permutations are independent for different $i$. Dividing the displayed count by $\prod_i k_i!$, we obtain \eqref{eq:coeff}. The coefficient in \eqref{eq:coeff} is positive for every $\mfk\in\Omega$, which implies that $\operatorname{supp}(P)=\Omega$.
\end{proof}

\begin{example}\label{ex:laminar-coefficient}
We illustrate both the coefficient formula and the role of the admissibility conditions. Let $P=(z_1+z_2)^2(z_1+z_2+z_3)$, so $\cL=\{[1,2],[1,3]\}$ with exponents $2$ and $1$.
For $\mfk=(2,1,0)$, we have $\wt(\mfk)=3$, $h_{[1,2]}(\mfk)=1$ and $h_{[1,3]}(\mfk)=0$. Hence, by \eqref{eq:coeff}, we have
\[
 [z_1^2z_2]P
 =\frac1{2!\,1!\,0!}\frac{3!}{1!}\frac{1!}{0!}=3.
\]
Indeed, the labeled copies are $z_1^{(1)},z_1^{(2)},z_2^{(1)}$. We can assign two of them to the two factors $z_1+z_2$ in $3\cdot2=6$ ways, and the remaining copy goes to $z_1+z_2+z_3$. Each choice of summands is counted $2!$ times because we can exchange the labels on the two copies of $z_1$. Thus, the coefficient is $6/2!=3$.

On the other hand, $\mfk=(1,0,2)$ satisfies $\wt(\mfk)=3$, $h_{[1,3]}(\mfk)=2$ but $h_{[1,2]}(\mfk)=-1$. Therefore, we have $[z_1z_3^2]P=0$, as stated by the support condition.
\end{example}

\subsection{Log-concavity along integer lines}
We first define the \emph{factorial weight functions} $V,W_a:\ZZ\to\RR_{\geq0}$ by
\begin{equation}\label{eq:VW}
 V(r)=\begin{cases}1/r!,&r\geq0,\\0,&r<0,\end{cases}
 \quad
 W_a(r)=\begin{cases}(r+a)!/r!,&r\geq0,\\0,&r<0,\end{cases}
 \quad(a\in\NN).
\end{equation}
Note that $W_0(r)$ equals $1$ for $r\geq0$ and $0$ otherwise. By Proposition~\ref{prop:coeff}, for every $\mfk\in\ZZ^m$ of total degree $d$, we have
\begin{equation}\label{eq:coeffzero}
 [\mfz^{\mfk}]P
 =\prod_{i=1}^m V(k_i)\prod_{I\in\cL}W_{a_I}(h_I(\mfk)).
\end{equation}
Indeed, $V$ is zero when a coordinate is negative, and $W_{a_I}$ is zero when an admissibility inequality fails.

\begin{lemma}\label{lem:factorials}
For any integers $r,s$ and $a\in\NN$, the sequences
\[
 \{V(r+sq)\}_{q\in\ZZ},\quad \{W_a(r+sq)\}_{q\in\ZZ}
\]
are log-concave and have interval support.
\end{lemma}
\begin{proof}
Both sequences are positive exactly when $r+sq\geq0$. Since this inequality defines an integer interval, both sequences have interval support. If $s=0$ and $a=0$, they are constant and the assertion is immediate. If $s<0$, we reverse the index $q$. Hence, we may assume that $s>0$ and $a>0$.
For consecutive indices in the support, we have
\[
 \frac{V(r+s(q+1))}{V(r+sq)}
 =\frac1{\prod_{\ell=1}^s(r+sq+\ell)}.
\]
This ratio is nonincreasing in $q$, which implies log-concavity whenever three consecutive terms are positive. For the second sequence, we write
\[
 W_a(r+sq)=\prod_{\ell=1}^a(r+sq+\ell)
\]
on its support. Each factor satisfies
\[
 (r+sq+\ell)^2
 -(r+s(q-1)+\ell)(r+s(q+1)+\ell)=s^2\geq0.
\]
Therefore, multiplying these inequalities proves log-concavity wherever three consecutive terms are positive. 
\end{proof}

\begin{proposition}[Log-concavity along integer lines]\label{prop:lines}
For any $\mfk^0,\mfv\in\ZZ^m$, the sequence
\[
 \{[\mfz^{\mfk^0+q\mfv}]P\}_{q\in\ZZ}
\]
is log-concave and has interval support. If $\mfv\neq\mathbf{0}$, it has finite support and is unimodal.
\end{proposition}
\begin{proof}
Set $\mfk(q)=\mfk^0+q\mfv$ and $a_q=[\mfz^{\mfk(q)}]P$ for $q\in\ZZ$. If $\mfv=\mathbf{0}$, then $\{a_q\}_{q\in\ZZ}$ is a nonnegative constant sequence. It is therefore log-concave, and its support is either empty or all of $\ZZ$. Then, we may henceforth assume that $\mfv\neq\mathbf{0}$.

Since $P$ is homogeneous of degree $d$, the condition $a_q\neq0$ implies
\begin{equation}\label{eq:linedegree}
 \wt(\mfk(q))=\wt(\mfk^0)+q\wt(\mfv)=d.
\end{equation}
If $\wt(\mfv)\neq0$, this equation has at most one integer solution $q$. If $\wt(\mfv)=0$ but $\wt(\mfk^0)\neq d$, it has no solution. In either case, $\{a_q\}_{q\in\ZZ}$ has at most one nonzero term, so all the assertions hold.
It remains to consider the case
$
 \wt(\mfv)=0, \wt(\mfk^0)=d.
$
Now, every $\mfk(q)$ has total degree $d$, so the coefficient formula \eqref{eq:coeffzero} applies for every integer $q$. We observe that
\[
 k_i(q)=k_i^0+qv_i,\quad
 h_I(\mfk(q))=h_I(\mfk^0)+q\wt_I(\mfv).
\]
Consequently, we obtain that 
\[
 a_q=\prod_{i=1}^m V(k_i^0+qv_i)
     \prod_{I\in\cL}W_{a_I}\bigl(h_I(\mfk^0)+q\wt_I(\mfv)\bigr).
\]
Each factor has the form $V(r+sq)$ or $W_a(r+sq)$ for fixed integers $r,s$ and $a\in\NN$. By Lemma~\ref{lem:factorials}, each is log-concave with interval support as a sequence in $q$. Their pointwise product $\{a_q\}_{q\in\ZZ}$ has the same two properties by Lemma~\ref{lem:closures}.
Since $\mfv\neq\mathbf{0}$, the sequence $\{a_q\}_{q\in\ZZ}$ has finite support and is therefore unimodal by Lemma~\ref{lem:ratio}.
\end{proof}

\begin{example}\label{ex:integer-line}
We take a laminar interval family $\cL=\{[1,2],[1,3]\}$ with $a_{[1,2]}=2$ and $a_{[1,3]}=4$. Then, we obtain
\[
 P=(z_1+z_2)^2(z_1+z_2+z_3)^4.
\]
With $\mfk^0=(0,1,5)$ and $\mfv=(1,0,-1)$, we obtain the integer line
\[
 \mfk(q)=\mfk^0+q\mfv=(q,1,5-q).
\]
Note that $\wt(\mfk(q))=6=d$ for every $q$. Moreover,
\[
 h_{[1,2]}(\mfk(q))=(q+1)-2=q-1,
 \quad h_{[1,3]}(\mfk(q))=6-6=0.
\]
Thus, $\mfk(q)\in\Omega$ exactly when $1\leq q\leq5$. By \eqref{eq:coeffzero}, the coefficient along this line is
\[
 [z_1^qz_2z_3^{5-q}]P
 =V(q)V(1)V(5-q)W_2(q-1)W_4(0)
 =\frac{24q(q+1)}{q!(5-q)!}\quad(1\leq q\leq5),
\]
and it is zero for every other integer $q$. The five nonzero coefficients in order are $2,12,24,20,6$. We can check log-concavity directly from $12^2\geq2\cdot24$, $24^2\geq12\cdot20$, and $20^2\geq24\cdot6$. Since its support is $[1,5]$, the sequence is also unimodal, as asserted in Proposition~\ref{prop:lines}.
\end{example}

\begin{remark}
Note that Proposition~\ref{prop:lines} requires only laminarity, so the members of $\cL$ need not be intervals. We will use the interval assumption later to analyze the combined coefficients under the substitution into fan coordinates.
\end{remark}

\section{Fan expansions and exponent fibers}\label{sec:fan}
In this section, all Laurent expansions are taken with respect to the fan initial cluster $\mfx$ in \eqref{eq:fantri}. We express cluster monomials using laminar interval products and study the associated exponent map and its fibers.
\subsection{The fan expansion formula}
We define
\begin{equation}\label{eq:u}
 u_i=p_{0i}\quad(1\leq i\leq n+2),\quad
 u_1=u_{n+2}=1,\quad u_{i+1}=x_i\quad(i\in[n]).
\end{equation}
We then set
\begin{equation}\label{eq:w}
 w_i=\frac1{u_i u_{i+1}}=\frac1{x_{i-1} x_{i}} \quad(i\in[n+1]).
\end{equation}
Note that the variables $z_i$ in Section~\ref{sec:laminar} are independent, whereas the Laurent monomials $w_i$ satisfy relations. For example,
\[
 \prod_{i\in[n+1],\ i\text{ odd}}w_i
 =\prod_{i\in[n+1],\ i\text{ even}}w_i.
\]
Indeed, $u_1=u_{n+2}=1$, and each remaining $u_j$ occurs exactly once in the denominator on either side. We define the \emph{substitution map} as 
\begin{equation}\label{eq:Phi}
 \Phi:\ZZ[z_1,\ldots,z_{n+1}]\longrightarrow
       \ZZ[x_1^{\pm1},\ldots,x_n^{\pm1}],\quad \Phi(z_i)=w_i.
\end{equation}
Thus, the products of the $z_i$ over odd and even indices have the same image under $\Phi$. Whenever distinct monomials occur in $P$ and have the same image, their coefficients add under the substitution.

\begin{lemma}[Fan expansion]\label{lem:fanexp}
For $1\leq a<b\leq n+2$, the Laurent expansion of $p_{ab}$ in the fixed fan initial cluster is
\begin{equation}\label{eq:fanexp}
 p_{ab}=u_a u_b\sum_{i=a}^{b-1}w_i.
\end{equation}
\end{lemma}
\begin{proof}
We prove the statement by induction on $b-a$. If $b=a+1$, both sides of \eqref{eq:fanexp} equal $1$. For $b\geq a+2$, we apply the Ptolemy relation to the cyclically ordered vertices $0,a,b-1,b$. Since $p_{b-1,b}=1$, we obtain
\[
 u_{b-1}p_{ab}=u_a+u_b p_{a,b-1}.
\]
Then, dividing by $u_a u_b u_{b-1}$ and using the induction hypothesis, we have
\[
 \frac{p_{ab}}{u_a u_b}
 =w_{b-1}+\frac{p_{a,b-1}}{u_a u_{b-1}}
 =w_{b-1}+\sum_{i=a}^{b-2}w_i
 =\sum_{i=a}^{b-1}w_i.
\]
This proves \eqref{eq:fanexp}.
\end{proof}

\begin{example}\label{ex:fan-p46}
We check the fan expansion formula in type $A_4$ for a single flip and recover the usual exchange relation. In the fan of Figure~\ref{fig:a4-fan-quiver}, we flip the diagonal $(0,5)$ to $(4,6)$. Here $x_4=p_{05}$, and we have $u_4=x_3$, $u_6=1$, $w_4=1/(x_3x_4)$, and $w_5=1/x_4$. By Lemma~\ref{lem:fanexp}, we have
\[
 p_{46}=u_4u_6(w_4+w_5)
 =x_3\left(\frac1{x_3x_4}+\frac1{x_4}\right)
 =\frac{1+x_3}{x_4}.
\]
Hence, $x_4p_{46}=1+x_3$, in agreement with the Ptolemy relation for the quadrilateral with vertices $0,4,5,6$.
\end{example}

\subsection{Compatible diagonals and laminar products}
We now show how compatibility of diagonals gives rise to laminar interval products in the fan expansions of cluster monomials. The following combinatorial observation is independent of the choice of initial triangulation.
\begin{lemma}\label{lem:laminararcs}
For a compatible family of diagonals $(a,b)$ with $1\leq a<b\leq n+2$, the intervals
\[
 I_{ab}=[a,b-1]\subseteq[n+1]
\]
form a laminar interval family.
\end{lemma}
\begin{proof}
Choose two distinct diagonals $(a,b)$ and $(c,d)$ from the family, and assume $a\leq c$. If $a=c$, their intervals have the same left endpoint and are nested.
Assume now that $a<c$. If $b\leq c$, then $I_{ab}\cap I_{cd}=\varnothing$. Otherwise, $a<c<b$. If $d>b$, the endpoints satisfy $a<c<b<d$, so the diagonals cross, contradicting compatibility. Therefore, $d\leq b$ and $I_{cd}\subseteq I_{ab}$. Thus, any two associated intervals are disjoint or nested, proving laminarity.
\end{proof}

\begin{proposition}\label{prop:reduction}
For every cluster monomial $M$, there exist $\mfeta\in\NN^n$ and a homogeneous polynomial $P$ of the form \eqref{eq:P}, indexed by a laminar interval family in $[n+1]$, such that
\begin{equation}\label{eq:reduction}
 M=\mfx^{\mfeta}\Phi(P).
\end{equation}
\end{proposition}
\begin{proof}
We write $M=\prod_{(a,b)\in\cT}p_{ab}^{r_{ab}}$ as in \eqref{eq:polygonmonomial}, where $r_{ab}\in\NN$ is the exponent of $p_{ab}$, and set $r_{ab}=0$ for diagonals outside $\cT$. In particular, $r_{0,i+1}$ is the exponent of $x_i=p_{0,i+1}$ in this product, taken to be zero if $x_i$ does not occur. By Lemma~\ref{lem:fanexp}, for each diagonal $(a,b)\in\cT$ with $a\geq1$ we have
\[
 p_{ab}^{r_{ab}}=(u_a u_b)^{r_{ab}}
               \left(\sum_{i=a}^{b-1}w_i\right)^{r_{ab}}.
\]
We therefore define
\[
 P(\mfz)=\prod_{(a,b)\in\cT,\ a\geq1}
       \left(\sum_{i=a}^{b-1}z_i\right)^{r_{ab}}.
\]
By definition, $\Phi(P)$ is obtained from $P$ by substituting $w_i$ for $z_i$. The remaining factors are powers of the initial variables and of the $u_a u_b$. By \eqref{eq:u}, their product is $\mfx^{\mfeta}$, where
\[
 \eta_i=r_{0,i+1}
   +\sum_{(a,b)\in\cT,\ a\geq1,\ i+1\in\{a,b\}}r_{ab}.
\]
Note that the boundary values $u_1=u_{n+2}=1$ contribute no initial variables. Hence, $\mfeta\in\NN^n$ and $M=\mfx^{\mfeta}\Phi(P)$.
By Lemma~\ref{lem:laminararcs}, the intervals $I_{ab}=[a,b-1]$ indexing $P$ form a laminar family $\mcL$. After adding $[n+1]$ with exponent zero if necessary, $P$ has the form \eqref{eq:P}. This includes $P=1$ when $M$ involves only initial variables.
\end{proof}

\subsection{The exponent map and its fibers}
We define the \emph{exponent map} $L:\ZZ^{n+1}\to\ZZ^n$ by
\begin{equation}\label{eq:L}
 L(\mfk)=(k_1+k_2,\ldots,k_n+k_{n+1})\in\ZZ^n.
\end{equation}
For $\mfk=(k_1,\dots,k_{n+1})\in\NN^{n+1}$, the definition of $\Phi$ implies
\begin{equation}\label{eq:monomialmap}
 \begin{aligned}
 \Phi(\mfz^{\mfk})
 &=\prod_{j=1}^{n+1}w_j^{k_j}
  =\prod_{j=1}^{n+1}(u_j u_{j+1})^{-k_j}\\
 &=\prod_{i=1}^n x_i^{-(k_i+k_{i+1})}
  =\mfx^{-L(\mfk)}.
 \end{aligned}
\end{equation}
The \emph{integer fiber of $L$ over $\mfb\in\ZZ^n$} is the set
\[
 \{\mfk\in\ZZ^{n+1}:L(\mfk)=\mfb\}.
\]
By \eqref{eq:monomialmap}, every nonnegative vector in this fiber corresponds to a monomial that maps to $\mfx^{-\mfb}$. The fiber may also contain vectors with negative coordinates, but these cannot be exponent vectors of monomials in $P$.

We now fix $P$ as in Proposition~\ref{prop:reduction}, with $d$ and $\Omega$ as in \eqref{eq:P} and \eqref{eq:Omega}. We write
\[
 c(\mfk)=[\mfz^{\mfk}]P\quad(\mfk\in\ZZ^{n+1}),\quad
 C(\mfb)=[\mfx^{-\mfb}]\Phi(P)\quad(\mfb\in\ZZ^n).
\]
We define the \emph{admissible fiber over $\mfb$} by
\[
 \Omega_{\mfb}=\{\mfk\in\Omega:L(\mfk)=\mfb\}.
\]
By Proposition~\ref{prop:coeff}, $\Omega_{\mfb}$ consists exactly of the exponent vectors in $P$ whose monomials have image $\mfx^{-\mfb}$. In particular, it is finite. Hence, by \eqref{eq:monomialmap}, we obtain
\begin{equation}\label{eq:fibersum}
 C(\mfb)=\sum_{\mfk\in\Omega_{\mfb}}c(\mfk).
\end{equation}
All summands are nonnegative, so $C(\mfb)\geq0$. Moreover, for $M=\mfx^{\mfeta}\Phi(P)$, we obtain
\begin{equation}\label{eq:shiftC}
 [\mfx^{\mfalpha}]M=C(\mfeta-\mfalpha).
\end{equation}
Thus, each coordinate slice of $M$ is a reversal and shift of a coordinate slice of $C$. By Lemma~\ref{lem:closures}, these operations preserve log-concavity and interval support, so it suffices to prove these properties for every coordinate slice of $C$.

\begin{lemma}\label{lem:fibereq}
Let $\mfb\in\ZZ^n$. Every integer solution of $L(\mfk)=\mfb$ is uniquely determined by $t=k_1\in\ZZ$ and is given by
\begin{equation}\label{eq:recsolution}
 k_i=(-1)^{i-1}t+\gamma_i(\mfb),\quad
 \gamma_i(\mfb)=\sum_{\ell=1}^{i-1}(-1)^{i-1-\ell}b_\ell
 \quad(1\leq i\leq n+1).
\end{equation}
Here $\gamma_1=0$, and the total degree of the solution is
\begin{equation}\label{eq:paritydegree}
 \wt(\mfk)=\begin{cases}
 t+b_2+b_4+\cdots+b_n,&n\text{ even},\\
 b_1+b_3+\cdots+b_n,&n\text{ odd}.
 \end{cases}
\end{equation}
\end{lemma}
\begin{proof}
We solve the equations recursively as $k_{i+1}=b_i-k_i$, starting from $k_1=t$. Thus, fixing $t$ uniquely determines all coordinates $k_i$. By induction, we obtain \eqref{eq:recsolution}, which gives a solution for every $t\in\ZZ$. If $n$ is even, we group the entries to obtain
\[
 \wt(\mfk)=k_1+(k_2+k_3)+\cdots+(k_n+k_{n+1})
       =t+b_2+b_4+\cdots+b_n.
\]
If $n$ is odd, we then have
\[
 \wt(\mfk)=(k_1+k_2)+(k_3+k_4)+\cdots+(k_n+k_{n+1})
       =b_1+b_3+\cdots+b_n.
\]
This proves \eqref{eq:paritydegree}.
\end{proof}

\begin{example}\label{ex:odd-fiber}
We illustrate how two distinct monomials of $P$ contribute to the same Laurent coefficient of $\Phi(P)$. We take $n=3$ with fan initial cluster variables $x_i=p_{0,i+1}$ for $1\leq i\leq3$. By Lemma~\ref{lem:fanexp}, we have $p_{15}=w_1+w_2+w_3+w_4$. Hence, the cluster monomial $M=p_{15}^2$ equals $\Phi(P)$ for $P=(z_1+z_2+z_3+z_4)^2$. Here $d=2$ and $\cL=\{[1,4]\}$, so $\Omega$ consists of all nonnegative vectors of total degree $2$.

By Lemma~\ref{lem:fibereq}, all integer solutions of $L(\mfk)=(1,1,1)$ are of the form
\[
 \mfk(t)=(t,1-t,t,1-t).
\]
Every solution satisfies $\wt(\mfk(t))=2$, and its coordinates are nonnegative if and only if $0\leq t\leq1$. Therefore, the admissible fiber consists of $\mfk(0)=(0,1,0,1)$ and $\mfk(1)=(1,0,1,0)$. The corresponding coefficients are $[z_2z_4]P=[z_1z_3]P=2$. Moreover,
\[
 \Phi(z_2z_4)=w_2w_4=\frac1{x_1x_2x_3}=w_1w_3=\Phi(z_1z_3).
\]
Therefore, by \eqref{eq:fibersum}, we have
\[
 [x_1^{-1}x_2^{-1}x_3^{-1}]\Phi(P)=C(1,1,1)=2+2=4.
\]
\end{example}

\section{Proof of the main theorem}\label{sec:mainproof}
We now prove Theorem~\ref{thm:main} by considering one coordinate slice at a time. We first distinguish between even and odd $n$. When $n$ is odd, we treat odd and even $j$ separately.
The sketch of proof is as follows.
\begin{itemize}
\item[$\bullet$] If $n$ is even, each Laurent coefficient comes from at most one monomial of the laminar product $P$. We therefore apply log-concavity along integer lines from Proposition~\ref{prop:lines}.
\item[$\bullet$] If $n$ is odd and $j$ is odd, the total degree condition forces the slice to have at most one nonzero term.
\item[$\bullet$] If $n$ is odd and $j$ is even, several coefficients of $P$ may contribute to the same Laurent coefficient. We express their sum as a convolution multiplied by a log-concave factor.
\end{itemize}
In each case, we establish log-concavity and interval support. Since every slice has finite support, unimodality then follows.

\begin{proof}[Proof of Theorem~\ref{thm:main}]
Let $M$ be any cluster monomial. By Proposition~\ref{prop:reduction}, we can write
\[
 M=\mfx^{\mfeta}\Phi(P),
\]
where $P$ is a laminar interval product in $n+1$ variables. We keep the notation $d,\cL,\Omega$ from Section~\ref{sec:laminar} and write
\[
 c(\mfk)=[\mfz^{\mfk}]P,
 \quad C(\mfb)=[\mfx^{-\mfb}]\Phi(P).
\]
By \eqref{eq:shiftC}, each coordinate slice of $M$ is a shift and reversal of a coordinate slice of $C$. These operations preserve log-concavity and interval support by Lemma~\ref{lem:closures}. Hence, it is enough to prove these two properties for every coordinate slice of $C$. All such slices have finite support, since $\Phi(P)$ is a Laurent polynomial, so unimodality will then follow from Lemma~\ref{lem:ratio}.

To study an arbitrary coordinate slice of $C$, we fix a direction $j\in[n]$ and integers $b_i$ for all $i\neq j$. For $q\in\ZZ$, we set $\mfb(q)=(b_1(q),\dots,b_n(q))$, where 
\begin{equation}\label{eq:bq}
 b_j(q)=q,\quad b_i(q)=b_i\ (i\neq j),\quad
 C_q=C(\mfb(q))
 =\sum_{\mfk\in\Omega,\ L(\mfk)=\mfb(q)}c(\mfk).
\end{equation}
We consider three cases according to the parities of $n$ and $j$.

\medskip\noindent\emph{Case 1. $n$ is even.}
For each $q$, we first solve $L(\mfk)=\mfb(q)$ together with $\wt(\mfk)=d$. Equivalently, it implies that 
\[
 k_i+k_{i+1}=b_i(q)\quad(1\leq i\leq n),
 \quad k_1+\cdots+k_{n+1}=d.
\]
Since $n$ is even, we leave $k_1$ aside and group the remaining coordinates in pairs. Then,
\[
 \begin{aligned}
 d&=k_1+(k_2+k_3)+(k_4+k_5)+\cdots+(k_n+k_{n+1})\\
  &=k_1+b_2(q)+b_4(q)+\cdots+b_n(q).
 \end{aligned}
\]
We write $k_i=k_i(q)$ for each coordinate to make its dependence on $q$ explicit. Thus,
\begin{equation}\label{eq:evenfirst}
 k_1(q)=d-\sum_{1\leq i\leq n,\ i\text{ even}}b_i(q).
\end{equation}
Starting with this value, we determine the remaining coordinates recursively by
\[
 k_{i+1}(q)=b_i(q)-k_i(q)\quad(1\leq i\leq n).
\]
This gives an integer vector satisfying $L(\mfk(q))=\mfb(q)$. Pairing its coordinates as above gives $\wt(\mfk(q))=d$. Moreover, the first coordinate is forced by \eqref{eq:evenfirst}, and the recursion determines every subsequent coordinate. Hence, this solution is unique.

Note that each $b_i(q)$ is either the fixed integer $b_i$ or $q$. By \eqref{eq:evenfirst}, we can write $k_1(q)$ as a constant plus an integer multiple of $q$. The recursion preserves this form, since it subtracts two expressions of the same form. By induction, every coordinate $k_i(q)$ is a constant plus an integer multiple of $q$. Collecting the constant terms and the coefficients of $q$, we can therefore write
\[
 \mfk(q)=\mfk^0+q\mfv
 \quad\text{with }\mfk^0,\mfv\in\ZZ^{n+1}.
\]
By uniqueness, the sum in \eqref{eq:fibersum} has at most one term, corresponding to $\mfk(q)$. If this vector is not admissible, the sum is empty and $c(\mfk(q))=0$ by Proposition~\ref{prop:coeff}. Thus, in all cases,
\[
 C_q=c(\mfk^0+q\mfv)\quad(q\in\ZZ).
\]
By Proposition~\ref{prop:lines}, we conclude that $\{C_q\}_{q\in\ZZ}$ is log-concave with interval support.

\medskip\noindent\emph{Case 2. $n$ is odd and $j$ is odd.}
If $C_q\neq0$, at least one admissible vector contributes to the sum in \eqref{eq:bq}. This vector satisfies $L(\mfk)=\mfb(q)$. By \eqref{eq:paritydegree}, we must have
\[
 d=b_1(q)+b_3(q)+\cdots+b_n(q)
   =q+\sum_{i\in[n],\ i\neq j,\ i\text{ odd}}b_i.
\]
The right side determines a unique value of $q$. Hence, $\{C_q\}_{q\in\ZZ}$ has at most one nonzero term and is therefore log-concave with interval support. Note that this case includes $n=1$.

\medskip\noindent\emph{Case 3. $n$ is odd and $j$ is even.}
Here $n+1$ is even and $2\leq j\leq n-1$. By \eqref{eq:paritydegree}, every solution of $L(\mfk)=\mfb(q)$ satisfies $\wt(\mfk)=b_1+b_3+\cdots+b_n$, which is independent of $q$ since $j$ is even. Thus, $\wt(\mfk)=d$ if and only if
\begin{equation}\label{eq:oddcompatibility}
 b_1+b_3+\cdots+b_n=d.
\end{equation}
If this equality fails, then $C_q=0$ for every $q$, and there is nothing to prove. We therefore assume that it holds. Our aim is to write $C_q$ as a convolution multiplied by a log-concave factor. We do this in four steps.

\smallskip\noindent\emph{Step 1. Express the coefficient as a sum over $t=k_1$.}
We set $c_i=\gamma_i(\mfb(0))$, using the notation in \eqref{eq:recsolution}. These are fixed integers. By the same formula, we have
\[
 \gamma_i(\mfb(q))=\sum_{\ell=1}^{i-1}(-1)^{i-1-\ell}b_\ell(q).
\]
If $i\leq j$, the sum does not include $b_j(q)=q$, since its indices satisfy $\ell\leq i-1<j$. All coordinates appearing in the sum are therefore fixed, so $\gamma_i(\mfb(q))=c_i$. If $i>j$, the sum includes $b_j(q)=q$, and this is its only term that depends on $q$. The coefficient of $q$ is $(-1)^{i-1-j}=(-1)^{i-1}$, since $j$ is even. Hence, $\gamma_i(\mfb(q))=c_i+(-1)^{i-1}q$. Substituting these expressions into \eqref{eq:recsolution}, we obtain all integer solutions in the form
\begin{equation}\label{eq:oddk}
 k_i(q,t)=\begin{cases}
 \eps_i t+c_i,&i\leq j,\\
 \eps_i(t+q)+c_i,&i>j,
 \end{cases}
 \quad \eps_i=(-1)^{i-1},\quad t\in\ZZ.
\end{equation}
Note that $c_1=0$ and $c_2=b_1$. Each solution satisfies $\wt(\mfk(q,t))=d$, and different values of $t=k_1$ give different solutions. We may therefore apply \eqref{eq:coeffzero} to every solution. If a vector is not admissible, at least one factor in that formula is zero. Thus, \eqref{eq:fibersum} becomes
\begin{equation}\label{eq:oddcoeffsum}
 C_q=\sum_{t\in\ZZ}
       \prod_{i=1}^{n+1} V(k_i(q,t))
       \prod_{I\in\cL}W_{a_I}(h_I(\mfk(q,t))).
\end{equation}
This sum is finite. Indeed, a nonzero summand requires $k_1=t\geq0$ and $k_2=b_1-t\geq0$, so only integers $0\leq t\leq b_1$ can contribute.

\smallskip\noindent\emph{Step 2. Determine which variable each factor depends on.}
We set $u=t+q$. By \eqref{eq:oddk}, the coordinates with $i\leq j$ depend only on $t$, while those with $i>j$ depend only on $u$. We next show that each $h_I(\mfk(q,t))$ is constant or depends on just one of $t,u,q$.

Let $I=[a,b]\in\cL$, where $1\leq a\leq b\leq n+1$. If $I\subseteq[1,j]$, then
\[
 \wt_I(\mfk(q,t))
 =\left(\sum_{i=a}^b\eps_i\right)t+\sum_{i=a}^b c_i.
\]
The signs $\eps_i$ alternate between $1$ and $-1$, so consecutive pairs cancel. Their sum is therefore $0$, $1$, or $-1$. Hence, $\wt_I$ is constant or a constant plus $t$ or $-t$. If $I\subseteq[j+1,n+1]$, the same argument applies with $u$ in place of $t$.

If $I$ is contained in neither $[1,j]$ nor $[j+1,n+1]$, its endpoints satisfy $a\leq j<b$. We split the interval into $[a,j]$ and $[j+1,b]$. By \eqref{eq:oddk}, we have
\[
 \wt_I(\mfk(q,t))
 =\left(\sum_{i=a}^j\eps_i\right)t
  +\left(\sum_{i=j+1}^b\eps_i\right)u
  +\sum_{i=a}^b c_i.
\]
We compute the two alternating sums separately. Since $j$ is even, the terms in $\sum_{i=a}^j\eps_i$ cancel in pairs if $a$ is odd. If $a$ is even, the terms from $a+1$ to $j$ cancel, leaving only $\eps_a=-1$. Thus,
\[
 \left(\sum_{i=a}^j\eps_i\right)t
 =\begin{cases}
 0,&a\text{ odd},\\
 -t,&a\text{ even}.
 \end{cases}
\]
Since $j+1$ is odd, the terms in the second sum cancel in pairs if $b$ is even. If $b$ is odd, the cancellation leaves only $\eps_b=1$. Hence,
\[
 \left(\sum_{i=j+1}^b\eps_i\right)u
 =\begin{cases}
 u,&b\text{ odd},\\
 0,&b\text{ even}.
 \end{cases}
\]
Combining these two calculations, we obtain the values of $\wt_I(\mfk(q,t))-\sum_{i\in I}c_i$ in the following table.
\begin{equation}\label{eq:paritytable}
 \begin{array}{c|cc}
  &b\text{ odd}&b\text{ even}\\ \hline
  a\text{ odd}&u&0\\
  a\text{ even}&u-t=q&-t
 \end{array}
\end{equation}
Subtracting the constant $D_I$ does not change which variable appears. This proves the claimed description of $h_I$.

We now group the intervals according to that description, writing
\[
 \cL=\cL_0\sqcup\cL_t\sqcup\cL_u\sqcup\cL_q.
\]
Here $\cL_0$ contains the intervals for which $h_I$ is constant, and the other three groups contain those for which $h_I$ depends on $t,u,q$, respectively. For each $I\in\cL$, we set
\[
 r_I=\sum_{i\in I}c_i-D_I.
\]
This is a fixed integer independent of $q$ and $t$. For each $I=[a,b]\notin\cL_0$, we define
\[
 s_I=\eps_b=(-1)^{b-1}
 =\begin{cases}
 1,&b\text{ odd},\\
 -1,&b\text{ even}.
 \end{cases}
\]
By the preceding calculations, $s_I$ is the coefficient of $t$, $u$, or $q$ in the corresponding case. Since $h_I=\wt_I-D_I$, we can therefore write
\begin{equation}\label{eq:partitionh}
 h_I(\mfk(q,t))=\begin{cases}
 r_I,&I\in\cL_0,\\
 r_I+s_I t,&I\in\cL_t,\\
 r_I+s_I u,&I\in\cL_u,\\
 r_I+s_I q,&I\in\cL_q.
 \end{cases}
\end{equation}

\smallskip\noindent\emph{Step 3. Collect the factors into three sequences.}
We put all factors depending on $t$ into $A(t)$, all factors depending on $u$ into $B(u)$, and the remaining factors into $H(q)$. More precisely, we define
\begin{align}
 A(t)&=\prod_{i=1}^{j}V(\eps_i t+c_i)
        \prod_{I\in\cL_t}W_{a_I}(r_I+s_I t),\label{eq:A}\\
 B(u)&=\prod_{i=j+1}^{n+1}V(\eps_i u+c_i)
        \prod_{I\in\cL_u}W_{a_I}(r_I+s_I u),\label{eq:B}\\
 H(q)&=\left(\prod_{I\in\cL_0}W_{a_I}(r_I)\right)
        \prod_{I\in\cL_q}W_{a_I}(r_I+s_I q).\label{eq:H}
\end{align}
Each factor is a function $V$ or $W_a$ evaluated at a constant plus an integer multiple of its variable. Therefore, Lemma~\ref{lem:factorials} applies to every factor, and Lemma~\ref{lem:closures} shows that the sequences $A(t)$, $B(u)$, and $H(q)$ are log-concave with interval support.

We now show that $A(t)$ and $B(u)$ can be nonzero for only finitely many integers $t$ and $u$. 
For $A(t)$, we consider the factors with indices $1$ and $2$, both of which occur in \eqref{eq:A} because $j\geq2$. Since $c_1=0$, $c_2=b_1$, $\eps_1=1$, and $\eps_2=-1$, these factors are
$
 V(t)$ and $V(b_1-t).
$
If $A(t)>0$, both factors must be positive. Hence, $t\geq0$ and $b_1-t\geq0$, giving
\[
 \supp(A)\subseteq[0,b_1].
\]
For $B(u)$, the factors with indices $j+1$ and $j+2$ both occur in \eqref{eq:B}, since $j\leq n-1$. As $j$ is even, we have $\eps_{j+1}=1$ and $\eps_{j+2}=-1$. Thus, these two factors are
$
 V(u+c_{j+1})$ and $V(c_{j+2}-u).
$
If $B(u)>0$, both factors must be positive. Therefore, $u+c_{j+1}\geq0$ and $c_{j+2}-u\geq0$, which implies that 
\[
 \supp(B)\subseteq[-c_{j+1},c_{j+2}].
\]
Thus, the sequences $A(t)$, $B(u)$, and $A(-t)$ have finite support, possibly empty.

\smallskip\noindent\emph{Step 4. Apply convolution and multiply by $H(q)$.}
By \eqref{eq:oddk}, \eqref{eq:partitionh}, and the definitions in \eqref{eq:A}--\eqref{eq:H}, we can write each term in \eqref{eq:oddcoeffsum} as
\[
 \begin{aligned}
 c(\mfk(q,t))
 &=\prod_{i=1}^{n+1}V(k_i(q,t))
   \prod_{I\in\cL}W_{a_I}(h_I(\mfk(q,t)))\\
 &=\underbrace{\prod_{i=1}^{j}V(\eps_i t+c_i)
   \prod_{I\in\cL_t}W_{a_I}(r_I+s_I t)}_{A(t)}\\
 &\quad{}\cdot\underbrace{\prod_{i=j+1}^{n+1}V(\eps_i(t+q)+c_i)
   \prod_{I\in\cL_u}W_{a_I}(r_I+s_I(t+q))}_{B(t+q)}\\
 &\quad{}\cdot\underbrace{\prod_{I\in\cL_0}W_{a_I}(r_I)
   \prod_{I\in\cL_q}W_{a_I}(r_I+s_I q)}_{H(q)}\\
 &=A(t)B(t+q)H(q).
 \end{aligned}
\]
This identity holds for all integer $q,t$, including when the coefficient is zero. Since $H(q)$ does not depend on $t$, we can take it outside the sum and obtain
\begin{equation}\label{eq:maincorrelation}
 C_q=H(q)\sum_{t\in\ZZ}A(t)B(t+q)
    =H(q)(\check A*B)_q,\quad \check A(t)=A(-t).
\end{equation}
By Lemma~\ref{lem:closures}, the reversed sequence $A(-t)$ is log-concave with finite interval support. Thus, by Proposition~\ref{prop:convolution}, the convolution sequence $\{(\check A*B)_q\}_{q\in\ZZ}$ has the same properties. Since $H(q)$ is log-concave with interval support, \eqref{eq:maincorrelation} and Lemma~\ref{lem:closures} show that $\{C_q\}_{q\in\ZZ}$ is also log-concave with interval support. Moreover, \eqref{eq:maincorrelation} implies $C_q=0$ whenever $(\check A*B)_q=0$, so this sequence has finite support regardless of the support of $H(q)$. This completes Case 3.

We have therefore proved that every coordinate slice of $C$ is log-concave with interval support. By \eqref{eq:shiftC}, the same property holds for $M$. Since each slice has finite support, unimodality follows from Lemma~\ref{lem:ratio}, which completes the proof.
\end{proof}

\section{Examples in types \texorpdfstring{$A_4$}{A4} and \texorpdfstring{$A_5$}{A5}}\label{sec:examples}
In this section, we illustrate the two main arguments in Section~\ref{sec:mainproof}. Indeed, log-concavity for cluster monomials of type $A_2$ \cite[Theorem~5.3]{CHS26}, and both log-concavity and unimodality in type $A_3$ \cite[Theorem~5.1]{Che26}, were proved using explicit Laurent expansions, Gauss hypergeometric functions and Jacobi polynomials. The examples below demonstrate how our laminar coefficient formula and fiber analysis provide a uniform approach in higher rank. In type $A_4$, we compute Laurent coefficients using the unique solution of the exponent equations with fixed total degree. In type $A_5$, we study a slice in the $x_2$ direction and write its coefficients in the form $H(q)(\check A*B)_q$. In each example, the monomial belongs to the cluster shown in Figure~\ref{fig:example-triangulations}, while the initial cluster is the fan centered at vertex $0$.
\begin{figure}[htbp]
\centering
\begin{tikzpicture}[line cap=round,line join=round,
  arc/.style={blue!70!black,thick},
  exponent/.style={font=\small,fill=white,inner sep=1pt,text=blue!70!black}]
  \begin{scope}
    \foreach \k in {0,...,6}{
      \coordinate (a\k) at ({180-360*(\k-1)/7}:1.65cm);
      \node[font=\small] at ({180-360*(\k-1)/7}:1.94cm) {$\k$};
    }
    \draw[thick] (a0)--(a1)--(a2)--(a3)--(a4)--(a5)--(a6)--cycle;
    \foreach \k/\lab in {3/t_3,4/t_4,5/t_5,6/t_6}
      \draw[arc] (a1)--node[exponent,pos=.58] {$\lab$}(a\k);
    \node at (0,-2.35) {Type $A_4$};
  \end{scope}
  \begin{scope}[xshift=6.2cm]
    \foreach \k in {0,...,7}{
      \coordinate (b\k) at ({180-360*(\k-1)/8}:1.65cm);
      \node[font=\small] at ({180-360*(\k-1)/8}:1.94cm) {$\k$};
    }
    \draw[thick] (b0)--(b1)--(b2)--(b3)--(b4)--(b5)--(b6)--(b7)--cycle;
    \foreach \k/\lab in {4/t_4,5/t_5,6/t_6,7/t_7}
      \draw[arc] (b1)--node[exponent,pos=.6] {$\lab$}(b\k);
    \draw[arc] (b2)--node[exponent,pos=.5] {$t_3$}(b4);
    \node at (0,-2.35) {Type $A_5$};
  \end{scope}
\end{tikzpicture}
\caption{The triangulations defining the clusters used in the examples. Each blue diagonal is labeled by its nonnegative exponent.}
\label{fig:example-triangulations}
\end{figure}
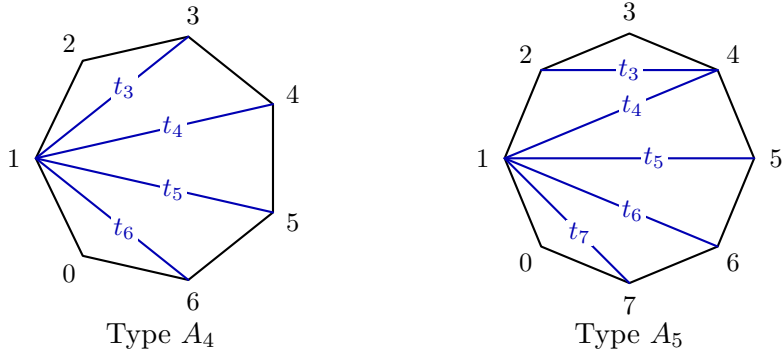

\subsection{Coefficient formula in type \texorpdfstring{$A_4$}{A4}}
We first make Case 1 of the main proof explicit. We take $n=4$ and choose the fan initial triangulation of the heptagon
\[
 \cT_0^{(4)}=\{(0,2),(0,3),(0,4),(0,5)\},
 \quad x_i=p_{0,i+1}\quad(i\in[4]).
\]
We take our monomial from the cluster corresponding to
\[
 \cT^{(4)}=\{(1,3),(1,4),(1,5),(1,6)\},
\]
which is shown on the left of Figure~\ref{fig:example-triangulations}. For $t_3,t_4,t_5,t_6\in\NN$, we set
\[
 M=p_{16}^{t_6}p_{15}^{t_5}p_{14}^{t_4}p_{13}^{t_3},
 \quad d=t_3+t_4+t_5+t_6.
\]
The four diagonals determine the laminar family
$
 \mcL=\{[1,2],[1,3],[1,4],[1,5]\},
$
whose members form a chain under inclusion. Then, by Proposition~\ref{prop:reduction}, the corresponding laminar product is
\begin{equation}\label{eq:exampleP}
 P=(z_1+\cdots+z_5)^{t_6}(z_1+\cdots+z_4)^{t_5}
   (z_1+z_2+z_3)^{t_4}(z_1+z_2)^{t_3}.
\end{equation}
Since $u_1=u_6=1$, $u_3=x_2$, $u_4=x_3$, and $u_5=x_4$, by Lemma~\ref{lem:fanexp}, we have 
\[
 M=\mfx^{\mfeta}\Phi(P),\quad \mfeta=(0,t_3,t_4,t_5).
\]

We first compute the coefficients of $P$. For $\mfk=(k_1,\dots,k_5)\in\NN^5$, we write
\[
 r=k_1+k_2,\quad s=k_1+k_2+k_3,\quad v=k_1+k_2+k_3+k_4.
\]
The values of $D_I$ on $[1,2]$, $[1,3]$, $[1,4]$, $[1,5]$ are $t_3$, $t_3+t_4$, $t_3+t_4+t_5$ and $d$, respectively. Hence, by \eqref{eq:Omega}, $\mfk$ is admissible if and only if
\[
 \wt(\mfk)=d,\quad r\geq t_3,\quad s\geq t_3+t_4,\quad v\geq t_3+t_4+t_5.
\]
Under these conditions, Proposition~\ref{prop:coeff} implies that 
\begin{equation}\label{eq:examplecoeff}
 [\mfz^{\mfk}]P
 =\frac{t_6!}{\prod_{i=1}^5k_i!}
   \frac{(v-t_3-t_4)!}{(v-t_3-t_4-t_5)!}
   \frac{(s-t_3)!}{(s-t_3-t_4)!}\frac{r!}{(r-t_3)!}.
\end{equation}
If $\mfk$ is not admissible, the coefficient is zero. Note that the admissibility conditions ensure that every factorial argument in \eqref{eq:examplecoeff} is nonnegative.

We next determine which coefficient of $P$ contributes to a given Laurent coefficient of $M$. For $\mfb=(b_1,\ldots,b_4)\in\ZZ^4$, we solve $L(\mfk)=\mfb$ together with $\wt(\mfk)=d$. By \eqref{eq:paritydegree}, we have $k_1=d-b_2-b_4$. Applying the recursion $k_{i+1}=b_i-k_i$, we obtain the unique integer solution
\begin{equation}\label{eq:rank4inverse}
 \begin{aligned}
 k_1&=d-b_2-b_4,&
 k_2&=b_1+b_2+b_4-d,\\
 k_3&=d-b_1-b_4,&
 k_4&=b_1+b_3+b_4-d,\\
 k_5&=d-b_1-b_3.&&
 \end{aligned}
\end{equation}
To compute $[\mfx^{\mfalpha}]M$, we set $\mfb=\mfeta-\mfalpha$ by \eqref{eq:shiftC} and substitute it into \eqref{eq:rank4inverse}. If all five coordinates are nonnegative and satisfy the admissibility inequalities above, the desired coefficient is given by \eqref{eq:examplecoeff}. Otherwise, it is zero.
Fix $j\in[4]$ and vary $\alpha_j=q$ while keeping the other exponents fixed. Substituting $\mfb=\mfeta-\mfalpha$ into \eqref{eq:rank4inverse}, we can write $\mfk(q)=\mfk^0+q\mfv$ for fixed $\mfk^0,\mfv\in\ZZ^5$. The preceding argument identifies the coefficient slice of $M$ with
\[
 \{[\mfz^{\mfk^0+q\mfv}]P\}_{q\in\ZZ},
\]
including the zero terms corresponding to nonadmissible vectors. By Proposition~\ref{prop:lines}, this slice is log-concave and has interval support. Its finite support then implies unimodality by Lemma~\ref{lem:ratio}.

\subsection{A coefficient slice in type \texorpdfstring{$A_5$}{A5}}
We next make the convolution argument in Case 3 explicit. We take $n=5$ and study the $x_2$ direction, so the rank is odd and the coordinate index is even. We choose the fan initial triangulation of the octagon
\[
 \cT_0^{(5)}=\{(0,2),(0,3),(0,4),(0,5),(0,6)\},
 \quad x_i=p_{0,i+1}\quad(i\in[5]).
\]
We take our monomial from the cluster corresponding to
\[
 \cT^{(5)}=\{(1,4),(1,5),(1,6),(1,7),(2,4)\}.
\]
This is the triangulation shown on the right of Figure~\ref{fig:example-triangulations}. 
For $t_3,t_4,t_5,t_6,t_7\in\NN$, we set
\[
 M=p_{17}^{t_7}p_{16}^{t_6}p_{15}^{t_5}p_{14}^{t_4}p_{24}^{t_3},
 \quad d=t_3+t_4+t_5+t_6+t_7.
\]
The corresponding laminar product is
\[
 P=(z_1+\cdots+z_6)^{t_7}(z_1+\cdots+z_5)^{t_6}
   (z_1+\cdots+z_4)^{t_5}(z_1+z_2+z_3)^{t_4}(z_2+z_3)^{t_3}.
\]
The associated laminar family is $\mcL=\{[2,3],[1,3],[1,4],[1,5],[1,6]\}$, whose members form the chain
\[
 [2,3]\subset[1,3]\subset[1,4]\subset[1,5]\subset[1,6].
\]
By Lemma~\ref{lem:fanexp}, we obtain
\[
 M=\mfx^{\mfeta}\Phi(P),\quad \mfeta=(t_3,0,t_3+t_4,t_5,t_6).
\]

We now fix integers $b_1,b_3,b_4,b_5$ and define the coefficient slice as in \eqref{eq:bq}, with $j=2$,
\[
 \mfb(q)=(b_1,q,b_3,b_4,b_5),\quad
 C_q=[\mfx^{-\mfb(q)}]\Phi(P)
     =[\mfx^{\mfeta-\mfb(q)}]M.
\]
For every solution of $L(\mfk)=\mfb(q)$, according to \eqref{eq:paritydegree}, we have $\wt(\mfk)=b_1+b_3+b_5$. Since $P$ is homogeneous of degree $d$, the slice is zero unless
$
 b_1+b_3+b_5=d.
$
We therefore assume that this equality holds. We set $t=k_1$ and $u=t+q$. By the recursion $k_{i+1}=b_i(q)-k_i$, we obtain all integer solutions in the form
\[
 \begin{aligned}
 k_1&=t,&k_2&=b_1-t,\\
 k_3&=u-b_1,&k_4&=b_1+b_3-u,\\
 k_5&=u+b_4-b_3-b_1,&
 k_6&=d-b_4-u.
 \end{aligned}
\]
We denote this solution by $\mfk(q,t)$, which satisfies $\wt(\mfk(q,t))=d$. Using these coordinates and $h_I(\mfk)=\wt_I(\mfk)-D_I$, we obtain
\[
 \begin{array}{c|c|c}
 I&D_I&h_I(\mfk(q,t))\\ \hline
 {[1,6]}&d&0\\
 {[1,5]}&t_6+t_5+t_4+t_3&u+b_4-t_6-t_5-t_4-t_3\\
 {[1,4]}&t_5+t_4+t_3&b_1+b_3-t_5-t_4-t_3\\
 {[1,3]}&t_4+t_3&u-t_4-t_3\\
 {[2,3]}&t_3&q-t_3
 \end{array}
\]
Thus, the interval groups from Step 2 of the main proof are
\[
 \cL_0=\{[1,6],[1,4]\},\quad \cL_t=\varnothing,\quad
 \cL_u=\{[1,5],[1,3]\},\quad \cL_q=\{[2,3]\}.
\]
By \eqref{eq:coeffzero}, we can therefore group the factors as in \eqref{eq:A}--\eqref{eq:H} to obtain
\[
 [\mfz^{\mfk(q,t)}]P=A(t)B(t+q)H(q),
\]
where
\begin{align*}
 A(t)&=V(t)V(b_1-t),\\
 B(u)&=V(u-b_1)V(b_1+b_3-u)
       V(u+b_4-b_3-b_1)V(d-b_4-u)\\
     &\quad{}\cdot W_{t_4}(u-t_4-t_3)W_{t_6}(u+b_4-t_6-t_5-t_4-t_3),\\
 H(q)&=t_7!\,W_{t_5}(b_1+b_3-t_5-t_4-t_3)W_{t_3}(q-t_3).
\end{align*}
Here, the constant factor $t_7!=W_{t_7}(0)$ corresponds to $I=[1,6]$, for which $h_I(\mfk(q,t))=0$.
From the factors $V(t)V(b_1-t)$ in $A(t)$ and $V(u-b_1)V(b_1+b_3-u)$ in $B(u)$, we obtain
\[
 \supp(A)\subseteq[0,b_1],\quad
 \supp(B)\subseteq[b_1,b_1+b_3].
\]
Note that $j=2$, so $b_1=-c_3=-c_{j+1}$ and $b_1+b_3=c_4=c_{j+2}$, as the support bound from Step 3 in the proof of Theorem~\ref{thm:main}.
Thus, $A(t)$ and $B(u)$ have finite support. By \eqref{eq:fibersum}, we have
\[
 C_q=H(q)\sum_{t\in\ZZ}A(t)B(t+q)
    =H(q)(\check A*B)_q,
 \quad \check A(t)=A(-t).
\]
As the Case 3 in the proof of Theorem~\ref{thm:main}, we conclude that the corresponding $x_2$ slice of $M$ is log-concave and unimodal, with interval support.

\section{Initial seeds and seed independence}\label{sec:scope}
The fan assumption enters our proof through the expansion in Lemma~\ref{lem:fanexp}, the laminar structure of compatible diagonals, and the exponent map \eqref{eq:L}. To extend the result to arbitrary initial seeds, one must understand how the coefficient inequalities behave under mutation.

\subsection{Transport of expansions and mutation}
We first consider Laurent expansions under isomorphisms.
\begin{proposition}\label{prop:transport}
Let $\varphi:\cA\to\cA'$ be a $\ZZ$-algebra isomorphism mapping cluster monomials bijectively onto cluster monomials. Suppose that the chosen initial clusters $\mfx$ and $\mfx'$ satisfy $\varphi(x_i)=x'_{\sigma(i)}$ for some $\sigma\in S_n$. Then, $\varphi$ preserves log-concavity, unimodality and the absence of internal zeros with respect to these initial clusters.
\end{proposition}
\begin{proof}
The isomorphism extends to the fraction fields. Applying this isomorphism to the Laurent expansion $M=\sum_{\mfalpha}c_{\mfalpha}\mfx^{\mfalpha}$, we obtain
\[
 \varphi(M)=\sum_{\mfalpha}c_{\mfalpha}\prod_{i=1}^n(x'_{\sigma(i)})^{\alpha_i}.
\]
Thus, the coefficient arrays agree up to a permutation of the coordinate directions, which preserves all three properties by Definition~\ref{def:multivariate}.
\end{proof}

To study seed independence, we fix a cluster monomial $M$ and compare its Laurent expansions in different initial clusters of the same algebra. We pass between these expansions by mutations. In Proposition~\ref{prop:transport}, we apply the isomorphism $\varphi$ to the Laurent expansion of $M$, sending $M$ to $\varphi(M)$ and each $x_i$ to $x'_{\sigma(i)}$. This invariance does not by itself establish the seed independence conjectured in \cite[Conjecture~5.11(2)]{Che26}.

The approach used in Proposition~\ref{prop:reduction} relies on the fan assumption. For a simple example, consider the non-fan initial triangulation $\{(0,2),(2,4),(0,4)\}$ of type $A_3$, with initial cluster $\mfx=(x_1,x_2,x_3)=(p_{02},p_{24},p_{04})$. The Ptolemy relations give
\[
 p_{03}=\frac{x_1+x_3}{x_2},\quad
 p_{13}=\frac{1+p_{03}}{x_1}=\frac{x_1+x_2+x_3}{x_1x_2}.
\]
The cluster variable $p_{13}$ cannot have the form in Proposition~\ref{prop:reduction} with the same substitution $\Phi(z_i)=1/(x_{i-1}x_i)$, where $x_0=x_4=1$. Indeed, setting $x_1=x_2=x_3=1$ would give $\prod_{I\in\cL}|I|^{a_I}=3$. Hence, exactly one interval of length $3$ has exponent $1$, and all other intervals with positive exponents are singletons. Setting $x_1=x_2=x_3=t$ then gives $\mfx^{\mfeta}\Phi(P)=t^\ell(t+2)$ for some $\ell\in\ZZ$, which cannot equal $p_{13}=3/t$. Thus, the expression in Proposition~\ref{prop:reduction} does not extend directly to arbitrary initial seeds, and a different argument for this case is needed; see also \cite{Che26}. 

\subsection{Seed independence}
We now restate the seed independence conjecture in the form relevant to our result.
\begin{conjecture}[{\cite[Conjecture~5.11(2)]{Che26}}]\label{conj:seed-independence}
Let $n\geq4$. For every cluster monomial $M$ of type $A_n$, the log-concavity and unimodality of its Laurent expansion are independent of the choice of initial seed.
\end{conjecture}
Together with Theorem~\ref{thm:main}, this conjecture would prove part~(1) of Conjecture~\ref{conj:cluster-monomials} for every initial seed. It would also imply the absence of internal zeros, since every nonnegative unimodal sequence has interval support.

\begin{remark}[Canonical bases and positivity]\label{rem:canonical-positivity}
The duality conjectures of Fock-Goncharov \cite[Section~4]{FG09} relate canonical functions to tropical points of the Langlands dual space, independently of a cluster chart. For cluster monomials on the cluster $\mathcal A$-variety, Conjecture~\ref{conj:seed-independence} asks for inequalities among Laurent coefficients in every chart, strengthening the requirement of nonnegative coefficients. Our theorem establishes these inequalities in fan charts.
\end{remark}

\begin{remark}[Supports and coefficient inequalities]\label{rem:tropical-coefficients}
Shen's work \cite[Theorem~1.1]{She14} describes the support of products in a canonical basis, whereas our conjecture concerns inequalities among Laurent coefficients; see Remark~\ref{rem:shen}. Convexity of the support alone does not imply log-concavity of the coefficients. The work of Goncharov-Shen \cite{GS15} motivates asking whether our coefficient formula admits a geometric interpretation that explains these inequalities under flips.
\end{remark}

\section*{Acknowledgements}
We would like to sincerely thank Yilin Wu for helpful discussions and valuable suggestions.
Z. Chen is also grateful to Nagoya University for its hospitality during his two-year visit. Z. Chen is supported by National Natural Science Foundation of China (Grant No. 124B2003) and China Scholarship Council (Grant No. 202406340022). Z. Sun is supported by the National Natural Science Foundation of China
(Grant No. 12471068).

%\newpage
\enlargethispage{\baselineskip}

\end{document}